\documentclass[11pt]{amsart} 
\usepackage{amssymb,amsmath,latexsym,enumerate,graphicx,bbm,mathptmx,lmodern,cite,ifthen,color}
\allowdisplaybreaks
\usepackage{cmbright}

\newtheorem{theorem}{Theorem} 

\newtheorem{lemma}[theorem]{Lemma}

\newtheorem*{bconjecture}{Bergeron's Conjecture}
\newtheorem{exam}{Example}

\newtheorem*{rem}{Remarks}

\newcommand\Def[1]{{\bf #1}}

\newcommand\GL{\operatorname{GL}}

\newcommand{\gauss}[2]{{\binom{#1}{#2}}_q}

\newcommand\commentout[1]{}

\makeatletter 
\newtheorem*{rep@theorem}{\rep@title}\newcommand{\newreptheorem}[2]{%
\newenvironment{rep#1}[1]{%
\def\rep@title{\bf #2 \ref{##1}}%
\begin{rep@theorem}}%
{\end{rep@theorem}}}
\makeatother
\newreptheorem{theorem}{Theorem}

\newcounter{teach}
\begin{document}

\title[Bergeron's Conjecture]{Bergeron's Conjecture \& A Tale of Two Binomial Coefficients}

\author[T. Amdberhan]{Tewodros Amdeberhan}

\author[M. Beck]{Matthias Beck}


\begin{abstract}
Bergeron's conjecture states that, if $1\leq a<b<c<d$ are integers with $ad=bc$, then one has the coefficient-wise
inequality $\gauss{b+c}b \ \geq \ \gauss{a+d}a$ among two Gaussian polynomials. It originated in algebraic
combinatorics and is wide open. The corresponding inequality for binomial coefficients (i.e., the case $q=1$) must be
known to experts, but we could not find it in the literature. We give two proofs, each generalizing the statement in
a separate direction. Binomial coefficients (and Gaussian polynomials) are fundamental combinatorial objects, and so
one naturally hopes to see a combinatorial proof of this inequality. However, this seems hard to come by. We
nevertheless give a combinatorial proof of a special case.
\end{abstract}



\date{4 July 2026}

\maketitle


\section{Introduction}

The \Def{Gaussian polynomial} (for positive integers $m \ge n$)
\[
  \gauss m n \, := \, \frac{ (1-q^m) (1-q^{ m-1 }) \cdots (1-q^{ m-n+1 }) }{ (1-q^n) (1-q^{ n-1 }) \cdots (1-q) } \, ,
\]
also known as a \Def{$q$-binomial coefficient}, appears in many mathematical instances and enumerate, e.g., inversions
in 0/1 words, integer partitions, subspaces of finite fields, and quantum groups.
Even such a basic mathematical object can hide surprisingly deep problems; maybe the most famous instance in this case
is the fact that a Gaussian polynomial is \Def{unimodal}, i.e., its coefficients increase up to some point and then
decrease. While this fact is old (going back to Cayley and Sylvester), a \emph{combinatorial} proof was found only in
1990, by O'Hara~\cite{ohara}.
Another such deep problem, still wide open, is the motivation for this note:

\begin{bconjecture}
{\it If $1\leq a<b<c<d$ are integers with $ad=bc$, then
$$\gauss{b+c}b \ \geq \ \gauss{a+d}a \, ,$$
where the inequality is understood coefficient-wise.}
\end{bconjecture}

By way of giving a brief historical background, the well-known \emph{Foulkes conjecture} (see, for instance ~\cite{bergeron2}) was generalized by Vessenes~\cite{vessenes}. 
She conjectured that 
\begin{align} \label{Vessenes}
(h_b\circ h_c) - (h_a \circ h_d)
\end{align}
is \emph{Schur positive} (expands with positive integer coefficients in the Schur basis $\{s_{\mu}\}_{\mu\vdash n}$ of
symmetric polynomials) whenever $a\leq b<c \leq d$, with $n=ad=bc$, and one writes $(h_n\circ h_k)$ for the
\emph{plethysm} of complete homogeneous symmetric functions. 
One can associate a direct combinatorial meaning to 
Vessenes' conjecture in the context of the representation theory
of $\GL(V)$. It would be natural to state that there is a surjective $\GL(V)$-module morphism the other way around
(which is also equivalent). Therefore each $\GL(V)$-irreducible occurs with smaller multiplicity in $S^a(S^d(V))$ than it does in $S^b(S^c(V))$, 
and the conjecture reflects this at the level of the corresponding characters 
(with Schur polynomials appearing as characters of irreducible representations). 

On the other hand, a well-known fact is that $(h_n\circ h_k)(1,q)=\binom{n+k}{k}_q$. Moreover, any non-zero evaluation of a Schur function 
at $1$ and $q$ is, for some $i<j$, of the form $q^i+q^{i+1}+\dots + q^j$. Exploiting these facts on the occasion of~\cite{bergeron}, and assuming that
Schur positivity of~\eqref{Vessenes} holds, Bergeron~\cite{bergeron2} underlined that the evaluation of the difference in \eqref{Vessenes}, at $1$ and $q$, 
would imply the above conjecture.

Again, Bergeron's conjecture is wide open, so it stands to reason that one looks at the special case $q=1$:

\begin{theorem}\label{thm:q=1_case}  If $1\leq a<b<c<d$ are integers with $ad=bc$, then
$$\binom{b+c}b \, \geq \, \binom{a+d}a \, .$$
\end{theorem} 

This result must be known to the experts, but we could not find a proof in the literature.
We give two proofs below (Theorems~\ref{real_case} and~\ref{q=1_thr_gamma}); one can be understood by a
(good) Calculus student, and the other uses Euler's gamma function.
In fact, in each case we prove a more general result, which might be of independent interest.

Binomial coefficients are mathematical constructs that are even more basic than Gaussian polynomials, and so the statement of
Theorem~\ref{thm:q=1_case} cries out for a combinatorial proof, which in turn would surely shed new light on
Bergeron's conjecture. However, finding such a proof seems hard.

As is well known, the quantity $\binom{b+c}b$ counts integer partitions whose Young diagram fit into a $b
\times c$ box---in other words, integer partitions with at most $b$ parts, each of which is of size at most~$c$.
(And the Gaussian polynomials give the corresponding generating function.)
Thus we may interpret Theorem~\ref{thm:q=1_case} via the following heuristic: given positive integers
$a<b<c<d$ such that the $a \times d$ and $b \times c$ rectangles have the same area, there are more
partitions fitting into the ``less extreme'' $b \times c$ box than into the $a \times d$ box.

At any rate, to share some productive combinatorial insights, we offer a combinatorial proof of a special case 
(setting $c=\beta a, d=\beta b$) of Theorem~\ref{thm:q=1_case} in Section~\ref{sec:combin1}.


\section{First Proof: Being "real"istic}

In what follows, for a real number $x$ and a non-negative integer $n$, we define the binomial coefficient in the standard way:
$$
  \binom{x}{n} \, = \, \frac{1}{n!} \, x(x-1) \cdots (x-n+1) \, .
$$
Furthermore, when both entries are continuous real variables, we extend this definition using Euler's gamma function~\cite[p.~235]{BN}. For real numbers $x$ and $y$ (chosen such that the arguments of the gamma function avoid non-positive integers), this is given by
$$
  \binom{x+y}{x} \, = \, \frac{\Gamma(x+y+1)}{\Gamma(x+1)\Gamma(y+1)} \, .
$$ 


\begin{theorem} \label{real_case} If $\beta\geq1$ is a real number and $b>a\geq1$ are integers, then
$$\binom{b+\beta a}b \, \geq \, \binom{a+\beta b}a \, .$$
\end{theorem}

Theorem~\ref{thm:q=1_case} follows from the special case $\beta=\frac{c}a=\frac{d}b$.

\begin{proof} Fix $a$ and $b$. Introduce the following functions of $\beta$: 
\[
R(a,b,\beta)
\, := \, \frac{\binom{b+\beta a}b}{\binom{a+\beta b}a}
\,  = \, \frac{a!}{b!}\, f(\beta)
\]
where
$$f(\beta) \, := \, \frac{(b+\beta a)(b-1+\beta a)\cdots(1+\beta a)}{(a+\beta b)(a-1+\beta b)\cdots (1+\beta b)} \, .$$
Because $a!, b!>0$, it suffices to verify that $f(\beta)$
increases for $\beta\geq1$ or, equivalently, that $\log f(\beta)$ increases for $\beta\geq1$. 
This, in turn, follows from a sequence of estimates.
\begin{align*}
\frac{d}{d\beta}\log f(\beta) 
\, &= \, \sum_{i=1}^b\frac{a} {i+\beta a}-\sum_{i=1}^a\frac{b} {i+\beta b} \\
 & = \sum_{i=a+1}^b \frac{a} {i+\beta a} + \sum_{i=1}^a\left(\frac{a} {i+\beta a}-\frac{b} {i+\beta b}\right)  \\
  &= \sum_{i=a+1}^b \frac{a} {i+\beta a}
 - (b-a)\sum_{i=1}^a\left(\frac{i}{(i+\beta a)(i+\beta b)}\right).
\end{align*}
For $1\le i\le a$, it is easy to see that
$$\frac{i}{i+\beta a} \, \leq \, \frac{a}{a+\beta a}=\frac1{\beta+1} \, .$$
Thus
$$\sum_{i=1}^a\frac{i}{(i+\beta a)(i+\beta b)}
 \, \leq \, \frac1{\beta+1}\sum_{i=1}^a \frac1{i+\beta b}
 \, \leq \, \frac{a}{(\beta+1)(1+\beta b)} \, .$$
On the other hand, 
$$a\sum_{i=a+1}^{b}\frac{1}{i+\beta a} \, \ge \, \frac{a(b-a)}{b+\beta a} \, .$$
Combining these inequalities gives
\begin{align*}
\frac{d}{d\beta}\log f(\beta)
\, & \, \ge a(b-a)\left(\frac{1}{b+\beta a}-\frac{1}{(\beta+1)(1+\beta b)}\right) \\
&= \, a(b-a)\, \frac{(\beta+1)(\beta b+1)-\beta a-b}{(b+\beta a)(\beta+1)(1+\beta b)}\, . 
\end{align*}
If we define $g(\beta):=(\beta+1)(\beta b+1)-\beta a-b)$, then
$g(1)=b-a+2>0$ and $g'(\beta)=2\beta b+b-a+1>0$
for $\beta\ge 1$. Hence $g(\beta)>0$, so this proves
$\frac{d}{d\beta}\log f(\beta)>0$
for all $\beta\ge 1$. Therefore $f(\beta)$, and hence $R(a,b,\beta)$, is increasing on $[1,\infty)$. Since $R(a,b,1)=1$, we conclude
$R(a,b,\beta)\ge 1,$ which is exactly $\binom{b+\beta a}{b}\ge \binom{a+\beta b}{a}$.
\end{proof}

We remark that this result can also be proved via the famous Chu--Vandermonde identity~\cite[(5.22)]{GKP}
$$\binom{X+Y}Z \ = \ \sum_{k\geq0} \binom{X}k\binom{Y}{Z-k} \, .$$


\section{Second Proof: the effect of $\Gamma$}

\begin{theorem} \label{q=1_thr_gamma}
Let $\theta > 0$ be a real number. Then the function
$$F(x) \ := \ \log \!\binom{x+\frac \theta x}{x}$$
is increasing on the interval $(0,\sqrt{\theta}]$.
\end{theorem}

Theorem~\ref{thm:q=1_case} follows from Theorem~\ref{q=1_thr_gamma} by setting $\theta = ad = bc$ and noting the equivalencies 
$a\leq \sqrt{\theta} \iff a\leq d$ and
$b\leq \sqrt{\theta} \iff b\leq c$,
and so, with $a \le b \leq \sqrt{\theta}$, the monotonicity of $F$ implies $F(b)\geq F(a)$, i.e., 
$$   \log \!\binom{b + c}{b}  \ \ge \ \log \!\binom{a + d}{a} \, .$$

\begin{proof} 
For $x>0$ we may write
$$F(x) \, = \,  \log \Gamma\!\left(x+\frac{\theta}{x}+1\right) - \log \Gamma(x+1)- \log \Gamma\!\left(\frac{\theta}{x}+1\right).$$
Let $y := \frac{\theta}{x}$ and $z := x + y + 1$, and recall that the \emph{digamma function} is $\psi = \frac{\Gamma'}{\Gamma}$.  
We compute 
\begin{align*}
F'(x) \ &= \ \left(1-\frac{y}{x}\right)\psi(z) - \psi(x+1)  + \frac{y}{x}\,\psi(y+1) \\
\ &= \ y \left(\frac{\psi(z) - \psi(x+1)}{y}  - \frac{\psi(z) - \psi(y+1)}{x} \right).
\end{align*}
Observe that
$$\psi''(x) \ = \ -2\sum_{k=0}^{\infty} \frac{1}{(x+k)^3} \ < \ 0 \, ,$$
that is, $\psi$ is concave. Hence the secant slope
$$ u \ \mapsto\ \frac{\psi(z)-\psi(u)}{z-u}$$
is decreasing in $u$. Therefore, if $x \le y$, then we gather the comparison
$$\frac{\psi(z)-\psi(x+1)}{z-(x+1)}   \;\ge\;  \frac{\psi(z)-\psi(y+1)}{z-(y+1)} \, .$$
But $z-(x+1) = y$ and $z-(y+1) = x$. Thus $F'(x)\ge 0$ whenever $x \le y$.  
Since $y=\frac \theta x$, the condition $x\le y$ is equivalent to
$$ x \le \frac{\theta}{x} \quad\Longleftrightarrow\quad   x^2 \le \theta \quad\Longleftrightarrow\quad
x\leq\sqrt{\theta} \, . $$
Hence $F(x)$ is increasing on the interval $(0,\sqrt{\theta}]$.
\end{proof}

\section{Some Combinatorics}\label{sec:combin1}

We already mentioned that combinatorial proofs for the above results are hard to come by, as much as they are desired.
To exhibit that not all is lost, we now build such a combinatorial proof for a special case of Theorem~\ref{real_case},
that is, $\beta\geq2$ is an integer (the case $\beta=1$ is clear).

\begin{lemma} \label{lemma1}
For integers $b> a \ge 1$,
$$ \binom{b+2a}b \ge \binom{a+2b}{a}.$$
\end{lemma}

\begin{proof} By the division algorithm, we write $b=na+r$ where $n\geq1$ and $a>r\geq0$ so that the assertion is tantamount to
$$ \binom{(n+2)a+r}{2a} \ge \binom{(2n+1)a+2r}{a}.$$
 Let $U$ be a set with $|U| = (n+2)a+r$, and let $V$ be a disjoint set with $|V| = (n-1)a+r$. Define $S = U \cup V$, so $|S| = (2n+1)a+2r$. 

Let's represent the two binomial coefficients as collections of subsets:

\begin{itemize}
    \item Let $\mathcal{A}$ be the collection of all $a$-element subsets of $S$, i.e., $\mathcal{A} = \{ A \subseteq S : |A| = a \}$. Then
    $$
    |\mathcal{A}|=\binom{(2n+1)a+2r}{a}.
    $$

    \item Let $\mathcal{B}$ be the collection of all $(2a)$-element subsets of $U$, i.e., $\mathcal B = \{ B \subseteq U : |B| = 2a \}$. Then
    $$
    |\mathcal{B}|=\binom{(n+2)a+r}{2a}.
    $$
\end{itemize}
We say that $A\in\mathcal{A}$ is \emph{linked to} $B\in\mathcal{B}$ if $A\cap U \subseteq B$.
Let $E$ denote the total number of such links.

\subsection*{Counting from the $\mathcal{B}$-side.}

Fix a subset $B\in\mathcal{B}$. For a subset $A\in\mathcal{A}$ to satisfy $A\cap U\subseteq B$, the set $A$ cannot contain any element of $U\setminus B$. Hence all elements of $A$ must be chosen from the disjoint union $B\cup V$. Since $|B|=2a$ and $|V|=(n-1)a+r$, we have $|B\cup V|=(n+1)a+r$.

Therefore the number of subsets $A$ linked to a fixed $B$ equals
$$
\binom{(n+1)a+r}a.
$$

Because this holds for every $B\in\mathcal{B}$, we gather that
$$E =|\mathcal{B}|\cdot \binom{(n+1)a+r}{a}
=
\binom{(n+2)a+r}{2a}\binom{(n+1)a+r}{a}.
$$

\subsection*{Counting from the $\mathcal{A}$-side.}
Fix a subset $A\in\mathcal{A}$ and let $k=|A\cap U|$. Since $|A|=a$, it must be that $0\leq k\leq a$.
To construct a subset $B\in\mathcal{B}$ linked to $A$, the set $B$ must contain all $k$ elements of $A\cap U$. 
The remaining $2a-k$ elements of $B$ must then be chosen from the remaining $(n+2)a+r-k$ members of $U$.
Hence the number of such subsets $B$ equals
$$\binom{(n+2)a+r-k}{2a-k}=\binom{(n+2)a+r-k}{na+r},$$
using symmetry of binomial coefficients. On the other hand, we have $(n+2)a+r-k \geq (n+1)a+r$. Using this observation together with the monotonicity $\binom{n}k\geq\binom{m}k$, when $n\geq m$, it follows that
$$\binom{(n+2)a+r-k}{na+r} \geq \binom{(n+1)a+r}{na+r}=\binom{(n+1)a+r}a.$$

Thus every subset $A\in\mathcal{A}$ links to at least
$$
\binom{(n+1)a+r}a
$$
subsets in $\mathcal{B}$. Consequently,
$$ E \geq |\mathcal{A}|\cdot \binom{(n+1)a+r}{a} = \binom{(2n+1)a+2r}a\binom{(n+1)a+r}a.$$

Combining the two counts for $E$, we arrive at the estimate
$$
\binom{(n+2)a+r}{2a}\binom{(n+1)a+r}a  \geq  \binom{(2n+1)a+2r}{a}\binom{(n+1)a+r}a.
$$
Canceling the positive factor $\binom{(n+1)a+r}a$ from both gives
$$
\binom{(n+2)a+r}{2a} \ge \binom{(2n+1)a+2r}a
$$
and hence the proof follows.
\end{proof}

The next result generalizes the above lemma while the spirit of part of its proof remains analogous.

\begin{theorem}
For integers $b> a \ge 1$ and integer $\beta\geq2$, 
$$\binom{b+\beta a}b \ge \binom{a+\beta b}{a}.$$
\end{theorem}

\begin{proof} For convenience, we write $b=na+r$ where $n\geq1$ and $a>r\geq0$ so that the assertion is tantamount to
$$ \binom{(n+\beta)a+r}{\beta a} \ge \binom{(\beta n+1)a+\beta r}{a}.$$
 Let $U$ be a set with $|U| = (n+\beta)a+r$, and let $V$ be a disjoint set with $|V| = (\beta-1)(n-1)a+(\beta-1)r$. Define $S = U \cup V$, so $|S| = (\beta n+1)a+\beta r$. 

Let's represent the two binomial coefficients as collections of subsets:

\begin{itemize}
    \item Let $\mathcal{A}$ be the collection of all $a$-element subsets of $S$, i.e., $\mathcal{A} = \{ A \subseteq S : |A| = a \}$. Then
    $$
    |\mathcal{A}|=\binom{(\beta n+1)a+\beta r}{a}.
    $$

    \item Let $\mathcal{B}$ be the collection of all $\beta a$-element subsets of $U$, i.e., $\mathcal B = \{ B \subseteq U : |B| = \beta a \}$. Then
    $$
    |\mathcal{B}|=\binom{(n+\beta)a+r}{\beta a}.
    $$
\end{itemize}
We say that $A\in\mathcal{A}$ is \emph{linked to} $B\in\mathcal{B}$ if $A\cap U \subseteq B$.
Let $E$ denote the total number of such links.

We proceed by induction on $\beta\geq2$. The base case $\beta=2$ is exactly the content of Lemma~\ref{lemma1}. 
For the induction step, assume the inequality in the statement of our theorem holds for $\beta - 1$.

\subsection*{Counting from the $\mathcal{B}$-side.}

Fix a subset $B\in\mathcal{B}$. For a subset $A\in\mathcal{A}$ to satisfy $A\cap U\subseteq B$, the set $A$ cannot contain any element of $U\setminus B$. Hence all elements of $A$ must be chosen from the disjoint union $B\cup V$. Since $|B|=\beta a$ and $|V|=(\beta-1)((n-1)a+r)$, we have $|B\cup V|=((\beta-1)n+1)a+(\beta-1)r$.

Therefore the number of subsets $A$ linked to a fixed $B$ equals
$$
\binom{((\beta-1)n+1)a+(\beta-1)r}a.
$$
Because this holds for every $B\in\mathcal{B}$, we gather that
\begin{align*}
E  \ &= \ |\mathcal{B}|\cdot \binom{((\beta-1)n+1)a+(\beta-1)r}a \\ 
&= \ \binom{(n+\beta)a+r}{\beta a} \binom{((\beta-1)n+1)a+(\beta-1)r}a \, .
\end{align*}

\subsection*{Counting from the $\mathcal{A}$-side.}
Fix a subset $A\in\mathcal{A}$ and let $k=|A\cap U|$. Since $|A|=a$, it must be that $0\leq k\leq a$.
To construct a subset $B\in\mathcal{B}$ linked to $A$, the set $B$ must contain all $k$ elements of $A\cap U$. 
The remaining $\beta a-k$ elements of $B$ must then be chosen from the remaining $(n+\beta)a+r-k$ members of $U$.
Hence the number of such subsets $B$ equals
$$\binom{(n+\beta)a+r-k}{\beta a-k}=\binom{(n+\beta)a+r-k}{na+r},$$
using symmetry of the binomials. On the other hand, we have $(n+\beta)a+r-k \geq (n+\beta-1)a+r$
based on $k\leq a$. Using this observation together with the monotonicity of the binomials $\binom{n}k\geq\binom{m}k$, when $n\geq m$, it follows that
$$\binom{(n+\beta)a+r-k}{na+r} \geq \binom{(n+\beta-1)a+r}{na+r}=\binom{(n+\beta-1)a+r}{(\beta-1)a}.$$

Thus every subset $A\in\mathcal{A}$ links to at least
$$
\binom{(n+\beta-1)a+r}{(\beta-1)a}
$$
subsets in $\mathcal{B}$. Consequently,
\begin{align*}
 E &\geq |\mathcal{A}|\cdot \binom{(n+\beta-1)a+r}{(\beta-1)a} = \binom{(\beta n+1)a+\beta r}a \binom{(n+\beta-1)a+r}{(\beta-1)a} \\
& \geq \binom{(\beta n+1)a+\beta r}a \binom{((\beta-1)n+1)a+(\beta-1)r}a
\end{align*}
where the last inequality is due to the induction hypothesis. 
Combining the two counts for $E$, from the $\mathcal{B}$-side and from $\mathcal{A}$-side, 
we arrive at the estimate
\begin{align*}
&\binom{(n+\beta)a+r}{\beta a} \binom{((\beta-1)n+1)a+(\beta-1)r}a \\
&\geq \binom{(\beta n+1)a+\beta r}a \binom{((\beta-1)n+1)a+(\beta-1)r}a \, .
\end{align*}
Canceling the positive factor $\binom{((\beta-1)n+1)a+(\beta-1)r}a$ from both gives
$$
\binom{(n+\beta)a+r}{\beta a} \ge \binom{(\beta n+1)a+\beta r}a
$$
and hence the proof follows.
\end{proof}

\bibliographystyle{amsplain}
\bibliography{bib}

\setlength{\parskip}{0cm} 

\end{document}